  \documentclass[12pt]{amsart}
\usepackage{kotex}
\usepackage{amsmath}
\usepackage{amsfonts}
\usepackage{amscd}

\usepackage{amssymb}
\usepackage{graphicx}
\usepackage{caption}
\usepackage{subcaption}
\usepackage{dsfont}
\usepackage{color}
\usepackage{hyperref}
\usepackage{bbm}
\usepackage{a4wide}
\usepackage{sseq}
\usepackage{tikz-cd}
\usepackage[abs]{overpic}

\usepackage{psfrag}
\usepackage{marvosym}
\usepackage{amssymb}
\usepackage{amsthm}
\usepackage{mathrsfs}
\usepackage{enumitem}
\usepackage{bbm}
 \usepackage{comment}
\usepackage{ytableau}
\usepackage{hyperref}

\usepackage{calligra}
\usepackage{mathrsfs}
\DeclareMathAlphabet{\mathcalligra}{T1}{calligra}{m}{n}
\DeclareFontShape{T1}{calligra}{m}{n}{<->s*[1.5]callig15}{}

\newtheorem{theorem}{Theorem}[section]

\newtheorem{lemma}[theorem]{Lemma}
\newtheorem{proposition}[theorem]{Proposition}

\theoremstyle{definition}
\newtheorem{definition}[theorem]{Definition}

\newtheorem{remark}[theorem]{Remark}

\newtheorem{theorem-definition}[theorem]{Theorem-Definition}

\numberwithin{equation}{section}

\newcommand{\CC} {\mathbb{C}}

\newcommand{\RR} {\mathbb{R}}

\newcommand{\ZZ} {\mathbb{Z}}

\newcommand {\shH} {\mathcal{H}}

\newcommand {\shL} {\mathcal{L}}

\newcommand {\shP} {\mathcal{P}}

\newcommand {\fot}  {\mathfrak{t}}

\renewcommand {\d} {\operatorname{d}}

\newcommand{\sExt}{\mathscr{E} \kern -1pt xt}

\newcommand {\Hom} {\operatorname{Hom}}
\newcommand {\sHom}{\mathscr{H}\kern-5pt\mathcalligra{om}}

\renewcommand {\Im} {\operatorname{Im}}

\renewcommand {\ker } {\operatorname{Ker}}
\newcommand {\Ker} {\operatorname{Ker}}

\newcommand{\Diff}{\operatorname{Diff}}

\newcommand {\rank} {\operatorname{rank}}

\newcommand {\rk} {\operatorname{rk}}

\newcommand {\Span} {\operatorname{Span}}

\newcommand{\sTor}{\mathscr{T} \kern -3pt or}

\newcommand {\vol} {\operatorname{vol}}

\begin{document}

\title[Limits of quantization from mixed to real polarizations on toric varieties]{Limits of quantization from mixed to real polarizations on toric varieties}

\author[Wang]{Dan Wang}
\address{Max Planck Institute for Mathematics, Vivatsgasse, 53111 Bonn, Germany}

\address{Department of Mathematics, The Chinese University of Hong
Kong, Shatin, Hong Kong}
\email{dwang116@link.cuhk.edu.hk}

\author[Yau]{Yutung Yau}
\address{Kavli Institute for the Physics and Mathematics of the Universe (WPI), The University of Tokyo Institutes for Advanced Study, The University of Tokyo, Kashiwa, Chiba 277-8583, Japan}
\email{yu-tung.yau@ipmu.jp}

\thanks{}
\maketitle

\begin{abstract}
Let $(M, \omega, J)$ be a $2n$-dimensional toric variety determined by a Delzant polytope $P$, whose $T^{n}$-symmetry determines a real polarization $\shP_{\RR}$. Let $K \subset T^{n}$ be a subtorus. By a construction due to Leung and the first author, the $K$-action induces a mixed polarization $\shP_{K}$.

This paper investigates the relationship between the quantum Hilbert spaces
\(\mathcal{H}_{K}\) and \(\mathcal{H}_{\RR}\) associated with the polarizations
\(\shP_{K}\) and \(\shP_{\RR}\). Starting from \(\shP_{K}\), we use an
imaginary-time flow to construct a one-parameter family of mixed polarizations
\(\shP_{K,t}\) on \(M\) interpolating between \(\shP_{K}\) and \(\shP_{\RR}\),
with \(\shP_{K,0}=\shP_{K}\) and
\(\lim_{t\to\infty}\shP_{K,t}=\shP_{\RR}\). For the corresponding quantum
Hilbert spaces \(\mathcal{H}_{K,t}\), we lift the imaginary-time flow to the
prequantum line bundle to obtain a \(T^{n}\)-equivariant isomorphism
\(\mathcal{H}_{K}\cong\mathcal{H}_{K,t}\). We finally show that
\(\mathcal{H}_{K,t}\) converges to \(\mathcal{H}_{\RR}\) as \(t\to\infty\).
\end{abstract}


\section{Introduction}

\emph{Geometric quantization} provides a mathematical formulation of the space of quantum states in a quantum mechanical system. 
Let $(M,\omega)$ be a symplectic manifold admitting a prequantum line bundle $(L,\nabla,h)$, where $L\to M$ is a complex line bundle equipped with a Hermitian metric $h$ and a Hermitian connection $\nabla$ whose curvature satisfies
\(
F_{\nabla}=- i \omega.
\)
Given a polarization $\shP\subset TM\otimes \CC$, the associated quantum space is
\[
\shH_{\shP}
=
\{s\in \Gamma(M,L)\mid \nabla_{\xi}s=0,\ \forall \xi\in \Gamma(M,\shP)\}.
\]

A fundamental question in this area is how the quantum space depends on the choice of polarizations. This dependence, particularly for K\"ahler and real polarizations, has been extensively studied in various settings \cite{BFMN, FMMN1, FMMN2, GS3, Hal, HK, KW}. Recently, new results have emerged concerning the existence of mixed polarizations and the structure of quantum spaces associated with mixed polarizations \cite{BFHMN,BHKMN,FHMN,LW2,LW1,LW3,MNP}. This motivates the previous (joint) work of the first author, \cite{W1} and \cite{MNGW} (with Mour\~ao--Nunes--Pereira), on the relationship between K\"ahler polarization and mixed polarizations on toric varieties.

This paper seeks to further explore the relationship between the quantum spaces associated with mixed and real polarizations. Throughout this paper, we work under the following standing assumption, which we denote by $(*)$:

let $(M, \omega, J)$ be a $2n$-dimensional toric variety determined by a Delzant polytope $P$, with associated moment map $\mu: M \twoheadrightarrow P$. Let $K \subset T^n$ be a $k$-dimensional subtorus and $\rho_K: K \to \mathrm{Diff}(M, \omega, J)$ be the subaction with corresponding moment map $\mu_K: M \to \mathfrak{t}_K^*$. We fix a splitting $Q \hookrightarrow T^{n}$ of the quotient torus $Q = T^{n} / K$ and let $\mu_Q$ be the corresponding moment map. Let $\phi: \mathfrak{t}_Q^* \to \mathbb{R}$ be a strictly convex smooth function. Denote by $X_\varphi$ the Hamiltonian vector field associated with $\varphi := \phi \circ \mu_Q$.

Under assumption $(*)$, by  \cite[Theorem 3.1]{LW1}, $M$ admits mixed polarization $\shP_{K}$ (see formula (\ref{def-mixp})) and real polarization $\shP_{\RR}$ (see formula (\ref{def-realp})). To study the relationship between the quantum spaces associated with these two polarizations, we apply the imaginary-time flow method of Mour\~ao and Nunes \cite{MN} to $\shP_K$. This yields a one-parameter family of mixed polarizations $\shP_{K,t}$ (Theorem~1.1), which interpolates between $\shP_K = \shP_{K, 0}$ and $\shP_{\RR}$ (Theorem~1.2).

\begin{theorem}(Theorem\ref{thm1}) Under assumption $(*)$, for any $t\ge0$,
$$\shP_{K,t}:=e^{-it\shL_{X_{\varphi}}}\shP_{K}$$
is a (singular) mixed polarization with $\rk_{\RR}(\shP_{K,t})=\rk_{\RR}(\shP_{K})$.
\end{theorem}
A difficulty in this result lie in demonstrating that the distribution $e^{-it\shL_{X_{\varphi}}}\shP_{K}$ exists on $M$. To resolve this issue,
we employ Sjammar's holomorphic slices theorem (see Theorem \ref{thm2-2}) to construct local model at each point of $M$, which allows us to apply the imaginary-time flow $e^{-it\shL_{X_{\varphi}}}$ to $\shP_{K}$.

\begin{theorem}(Theorem \ref{plim})
Under assumption $(*)$, we have:
$$\lim_{t \rightarrow \infty} \shP_{K,t}=\shP_{\RR}.$$
\end{theorem}
The key ingredient in the proof is 
the linear-algebraic Lemma~\ref{lem-com}, which shows that, under suitable constant-dimension assumptions, taking limits of subspaces commutes with both intersection with and addition of  a fixed subspace. Applying this fiberwise, we deduce the limit of  $\shP_{K,t}$ from the known limit of the Kähler polarizations $\shP_t$.

Let $L$ be the $T^{n}$-equivariant prequantum line bundle over $(M,\omega)$ determined by the polytope $P$, let $\hat{\varphi}$ be the quantum operator associated with $\varphi$ (see \eqref{eq2-0-9}), and let $\shH_{K,t}$ be the quantum space associated with the polarization $\shP_{K,t}$. We show that $e^{t\hat{\varphi}}$ can be applied to $\shP_{K}$-polarized distributional sections $\delta_{K}^{m}$ (see formula \ref{dis-basis0}); see Definition~\ref{def-imaginary-flow}. Moreover,
\begin{theorem}(Theorem \ref{imaginary-flow})
Under assumption $(*)$, for any $t >0$,  $e^{t\hat{\varphi}} : \shH_{K,0} \rightarrow \shH_{K,t}$ is a $T^{n}$-equivariant linear isomorphism.
\end{theorem}

For each integral point $m\in P\cap\mathfrak t_{\mathbb Z}^{*}$, let
$\delta^m\in\Gamma_c(M,L^{-1})'$ be the canonical basis element of
quantum space $\shH_{\mathbb R}$ associated with $\shP_{\RR}$ supported on the Bohr--Sommerfeld fiber
$\mu^{-1}(m)$, as defined in \eqref{eqn:delta^m(tau)}.
Let
$\delta_{K,t}^{m}\in\Gamma_c(M,L^{-1})'$ be the canonical basis element
of $\shH_{K,t}$, obtained by imaginary-time flow,
$ 
\delta_{K,t}^{m}=e^{t\widehat{\varphi}}\delta_K^m,
$ 
and described explicitly in \eqref{dis-basis}.

The following result relates the quantum spaces $\shH_{K}(=\shH_{K,0})$ and $\shH_{\mathbb{R}}$ by showing that, after a natural normalization, the mixed-polarized basis element $\delta_{K,t}^{m}$ converges to the real-polarized basis element $\delta^{m}$ as $t \to \infty$.

\begin{theorem}(Theorem \ref{thm-hlim})
Under assumption $(*)$, for any $m \in P \cap \fot^{*}_{\ZZ}$, the family of
distributional sections $\{C_{t}^{m}\delta_{K,t}^{m}\}_{t>0}\subset
\Gamma_{c}(M,L^{-1})'$ converges weakly to $C^{m}\delta^{m}$, that is: for
every test section $\tau\in \Gamma_{c}(M,L^{-1})$, one has
\begin{equation}\label{eq:weakconv}
\lim_{t\to\infty} C_{t}^{m}\delta_{K,t}^{m}(\tau)= C^{m}\delta^{m}(\tau).
\end{equation}
Here, $C_{t}^{m}$ is as in \eqref{eq-famconst}, and $C^{m}$ is the constant
defined by $C^{m}=|\sigma^{m}|(m)$.
\end{theorem}






\section{Acknowledgements}
We would like to thank Qingyuan Jiang and Ki Fung Chan for their helpful discussions.
D. Wang is grateful to the Max Planck Institute for Mathematics (MPIM) in Bonn for their hospitality and financial support, where part of this work was carried out. The work of D. Wang was partially supported by a grant from the Research Grants Council of the Hong Kong Special Administrative Region, China (Project No. CUHK 2130983).
The work of Y. Yau was supported by World Premier International Research Center Initiative (WPI), MEXT, Japan.

{\color{blue}} 

\section{Preliminaries}
\subsection{Toric K\"ahler geometry}

Let $P$ be a Delzant polytope, and let $(M,\omega_P)$ be the associated symplectic toric manifold with moment map $\mu$. The points of $M$ where the torus action is free form a dense open subset $\mathring{M} = \mu^{-1}(\mathring{P})$, where $\mathring{P}$ is the interior of $P$. Moreover, $\mathring{M} \cong \mathring{P} \times T^n$ as symplectic manifolds, with the latter endowed with the symplectic structure inherited from $\mathbb{R}^n \times T^n$. Toric Kähler structures can be described with the aid of {\em symplectic potentials} \cite{Gui2, Ab2}. 
In the action-angle coordinates $(x,\theta)$ on $\mathring{P} \times T^n$, $P$ is given by the linear inequalities
\begin{equation}
\label{eqn:l_r}
l_r(x) = \langle x, \nu_r\rangle +\lambda_r\geq 0, \quad r=1\, \dots, d,
\end{equation}
where $\nu_r$ is the primitive inward pointing normal to the $r$th facet of $P$. 
The function
\begin{align}\label{po2-0}
    g_P(x) = \frac{1}{2}\sum_{r=1}^d l_r(x)\log l_r(x),
\end{align}
defined for $x \in \mathring{P}$, endows $M$ with a torus-invariant compatible complex structure $J_P$ given by the block matrix
\begin{align}\nonumber
    J_P =
    \begin{bmatrix}
    0 & -G_P^{-1} \\
    G_P & 0
    \end{bmatrix},
\end{align}
where $G_P$ is the Hessian of $g_P$ \cite{Gui2}. As shown by Abreu in \cite{Ab2}, this is not the only compatible toric complex structure. More precisely,
\begin{theorem}[\cite{Ab2}]
Let $(M,\omega)$ be the symplectic toric manifold associated with a Delzant polytope $P$. If $J$ is a compatible toric complex structure, then
\begin{align}\label{com2-3}
    J =
    \begin{bmatrix}
    0 & -G^{-1} \\
    G & 0
    \end{bmatrix},
\end{align}
where $G$ is the Hessian of
\begin{align}\label{po2-1}
    g = g_P + h,
\end{align}
with $g_P$ being defined as in \ref{po2-0} and $h \in C^\infty(P)$. Moreover, $G$ is positive-definite on $\mathring{P}$ and satisfies the regularity condition
\begin{align}\nonumber
    \det G = \left(\delta(x)\prod_{r=1}^d l_r(x)\right)^{-1},
\end{align}
where $\delta$ is a smooth and strictly positive function on $P$.

Conversely, given a function $g$ as in \ref{po2-1} satisfying the above conditions, $J$ as defined in \ref{com2-3} is a compatible toric complex structure on $(M,\omega)$.
\end{theorem}
A function $g$ satisfying these properties is called a \emph{symplectic potential}.

Let $J$ be a compatible complex structure induced by a symplectic potential $g$ on $(M,\omega)$. Then $(\mathring{M},J) \cong \mathring{P} \times T^n$ as Kähler manifolds, where the complex structure on the latter is induced from $\mathbb{C}^n$ via the  biholomorphism
\begin{align}
    \mathring{P} \times T^n &\xrightarrow{\cong} (\mathbb{C}^*)^n \\
    (x,\theta) &\mapsto w = (e^{y_1+i\theta_1}, \dots, e^{y_n+i\theta_n}),
\end{align}
where $y_j = \partial g/\partial x^j$. The assignment $x \mapsto y = \partial g/\partial x$ is an invertible Legendre transform, with inverse given by $x = \partial h/\partial y$, where $h = x(y)\cdot y - g(x(y))$ is a Kähler potential given in terms of $g$. 

Associated with $P$ there is the line bundle $L = \mathcal{O}(D) \to M$, where
\begin{align}\nonumber
D = -\sum_{r=1}^d \lambda_rD_r,
\end{align}
with each $\lambda_r$ a non-negative integer, and $D_r$ the divisor given by the inverse image under the moment map of the facet of $P$ defined by $l_r = 0$. With this in mind, the meromorphic section $\sigma_D$ of $L$, whose divisor is $D$, trivializes $L$ over $\mathring{M}$. Under this trivialization, the divisor of the meromorphic function whose expression in the dense open subset is given by
\begin{align}\label{eq:monomials}
    w^m = w_1^{m_1}\cdots w_n^{m_n}, \quad m = (m_1, \dots, m_n) \in \mathbb{Z}^n,
\end{align}
satisfies
\begin{align}\nonumber
    \mathrm{div}(w^m) = \sum_{r=1}^d \langle m, \nu_r \rangle D_r.
\end{align}
Holomorphic sections of $L$ are generated by the sections which, under the trivialization $\sigma_D$, are given by monomials as in \eqref{eq:monomials} with effective divisors, i.e.
\begin{align}
    H^0(M,L) = \Span_{\CC}\{w^m \sigma_D : m \in \mathbb{Z}^n, \mathrm{div}(w^m\sigma_D) \geq 0\},
\end{align}
and $\mathrm{div}(w^m\sigma_D) \geq 0$ precisely when $l_r(m) = \langle m, \nu_r \rangle - \lambda_r \geq 0$ for $r = 1, \dots, d$. Consequently, the assignment $m \mapsto w^m \sigma_D$ defines a bijection between the set of integral points of $P$ and a toric basis of $H^0(M,L)$.

\subsection{Polarizations on symplectic manifolds}

A step in the process of geometric quantization is to choose a polarization.
We first recall the definitions of distributions and polarizations on
symplectic manifolds $(M,\omega )$ (See \cite{Wo}).

\begin{definition}
\label{def4-1}A \emph{complex distribution} $\mathcal{P}$ on a manifold $M$ 
is a complex sub-bundle of the complexified tangent bundle $TM\otimes 
\mathbb{C}$. 
When $\left( M,\omega \right) $ is a $2n$-dimensional symplectic manifold, such a $\mathcal{P}
$ is called a \label{def4-2}  \emph{polarization} if it satisfies the
following conditions:
\begin{enumerate}
\item  $\mathrm{rk}_{\RR}\left( \mathcal{P}\right):=\mathrm{rank}(\mathcal{P}\cap \overline{
\mathcal{P}}\cap TM)$ is constant;
\item $\mathcal{P}$ and $\mathcal{P}+\bar{\mathcal{P}}$ are involutive; and
\item for every $x \in M$, $\mathcal{P}_{x} \subset T_{x}M \otimes {
\mathbb{C}}$ is Lagrangian.
\end{enumerate}
Furthermore, $\shP$ is called
\begin{enumerate}
\item  [$\cdot$] a {\em real polarization} if $\mathcal{P}=\overline{\mathcal{P}}$,
i.e. $\mathrm{rk}_{\RR}\left( \mathcal{P}\right) = n$; 
\item  [$\cdot$] a {\em K\"ahler polarization} if $\mathcal{P}\cap \overline{
\mathcal{P}}=0$, i.e. $\mathrm{rk}_{\RR}\left( \mathcal{P}\right) =0$; 
\item [$\cdot$] a {\em mixed polarization} if 
$0<\mathrm{rk}_{\RR}\left( \mathcal{P}\right) < n$. 
\end{enumerate}
\end{definition}

It is easy to see that there exist K\"ahler polarization and singular real polarization defined by moment map on toric varieties. To study singular polarizations $\shP_{\mathrm{mix}}$ on toric manifolds, we recall the definitions of singular polarizations and smooth section of singular polarizations as follows (refer to section 3.1 of \cite{LW1} or Appendix of \cite{LW3}). 
\begin{definition}\label{def4-3}
A {\em singular complex distribution} on $M$ is a subset $\shP \subset  TM \otimes\CC$ which satisfies:
$\shP_{p} $ is a vector subspace of $ T_{p}M \otimes \CC$, for all points $p \in M$.
\end{definition}

\begin{definition}\label{def4-4}
Let $\shP$ be a singular complex distribution on $M$. For any open subset $U$ of $M$, {\em the space of smooth sections of $\shP$ on $U$} is defined by the smooth section of $TM\otimes \CC$ with value in $\shP$, that is,
$$ \Gamma(U, \shP) = \{ v \in \Gamma(U, TM \otimes \CC) \mid  v_{p} \in \shP_{p}, \forall  p\in U \}.$$
\end{definition}

\begin{remark} A singular polarization studied in this paper refers to a polarization with mild singularity in the sense that it is only singular outside an open dense subset $\check{M}\subset M$. Under our setting, we review the definition of  involutive distributions as follows.
\end{remark}

\begin{definition}\label{def4-5}
A singular complex distribution $\shP$ on $M$ is {\em involutive} if it satisfies
$$[u,v] \in \Gamma(M,\shP), \mathrm{~for~ any~} u,v \in \Gamma(M, \shP).$$
\end{definition}

\begin{definition}\label{def4-7} Let $(M, \omega)$ be a symplectic manifold. A singular complex distribution $\shP$ on $M$ is called a {\em singular polarization on $(M, \omega)$} if it satisfies the following conditions:
\begin{enumerate} [label = (\alph*)]
\item $\shP$ and $\shP+\bar{\shP}$ are involutive; and
\item there is an open dense subset $\check{M} \subset M$ such that $\shP|_{\check{M}}$ is a polarization on $(\check{M}, \omega |_{\check{M}})$ in the sense of Definition \ref{def4-1}.
\end{enumerate}
\end{definition}

\begin{definition} \label{def4-8}
Let $\shP$ be a singular polarization on an $2n$-dimensional symplectic manifold $(M, \omega)$ which is smooth on an open dense subset $\check{M} \subset M$. $\shP$ is called 
 \begin{enumerate}
 \item a {\em singular real polarization} if $\shP = \overline{\shP}$ on $\check{M}$;
 \item a {\em singular mixed polarization} if $0 < \rank(\shP \cap \overline{\shP} \cap TM)|_{\check{M}} < n$.
 \end{enumerate}
\end{definition}

\subsection {Sjamaar's holomorphic slices} In this subsection, we recall Sjamaar's work (see \cite{Sj}) on the
existence of holomorphic slices. Let $G$ be a compact, connected Lie group.

\begin{theorem}
\label{thm2-2}\cite[theorem 1.12]{Sj} Let $M$ be a K\"{a}hler
manifold with a Hamiltonian $G$-action by holomorphic isometries. Given any
point $p$ lying on an isotropic $G$-orbit in $M$, there exists a slice at $p$
for the $G^{\mathbb{C}}$-action. 
\end{theorem}

We recall the definition of slices.

\begin{definition}
A slice at $p \in M$ for the $G^{\mathbb{C}}$-action is a locally closed
analytic subspace $S$ of $M$ with the following properties:

\begin{enumerate}
\item $p \in S$;

\item the saturation $G^{\mathbb{C}}S$ of $S$ is open in $M$;

\item $S$ is invariant under the action of the stabilizer $(H^{\mathbb{C}
})_{p}$;

\item the natural $G^{\mathbb{C}}$-equivariant map from $G^{\mathbb{C}}\times_{(H^{\mathbb{C}})_{p}} S$ into $M$, which sends $[g, y]$ to the point $g\cdot y$, is an analytic isomorphism onto 
$G^{\mathbb{C}}S$.
\end{enumerate}
\end{definition}

\begin{remark}
The condition that $p$ lying on an isotropic $G$-orbit in $M$ is equivalent
to $\mu \left( p\right) \in \mathfrak{g}^{\ast }$ is fixed under the
coadjoint action of $G$ \cite[theorem 1.12]{Sj}. Therefore every orbit is
isotropic if $G=T^{n}$ is Abelian, as in our situation. 
\end{remark}

\subsection{Quantum Hilbert spaces associated with mixed polarizations}
 Let $(M, \omega, J)$ be a 2n-dimensional toric variety determined by a Delzant polytope $P$, with associated moment map $\mu: M \twoheadrightarrow P $.
 Consider the action of a Hamiltonian $k$-dimensional subtorus $K \subset T^{n}$, denoted $\rho_{K}: K \rightarrow \Diff(M, \omega, J)$, with corresponding moment map $\mu_{K}: M \rightarrow    \fot_{K}^{*}$. According to \cite[Theorem 3.8]{LW1} (see also \cite{W1}), a polarization $\shP_{K}$ can be constructed as follows:
 \begin{equation}\label{def-mixp}
 \shP_{K}=\left((\ker d\mu_{K} \otimes \CC)\cap \shP_{J} \right) \oplus (\Im d\rho_{K}\otimes \CC).
 \end{equation}
 In general, $\shP_{K}$ defines a (possibly singular) mixed polarization, with $\mathrm{rk}_{\RR}({\shP_{K}}|_{\mathring{X}})=k$ (see \cite[Theorem 1.1]{LW1}). 
 When $k=n$, the polarization
 \begin{equation}\label{def-realp}
 \left((\ker d\mu_{T^{n}} \otimes \CC)\cap \shP_{J} \right) \oplus (\Im d\rho_{T^{n}}\otimes \CC)
 \end{equation}
 defines a singular real polarization, which we denote by $\shP_{\RR}$. On the open dense subset $\check{M}$ of $M$, $\shP_{\RR}$ coincides with the real polarization given by $\Ker d\mu$ studied in \cite{BFMN}. Let $(L, h, \nabla)$ be the prequantum line bundle on $(M, \omega, J)$ determined by the moment polytope $P$ with curvature $F_{\nabla} = -i \omega$. The quantum space $\shH_{K}$ associated with $\shP_{K}$ is defined by
$$ 
\shH_{K}=
\Bigl\{\delta\in \Gamma_{c}(M,L^{-1})' \ \Big|\ 
\nabla_{\xi}\delta=0 \text{ for all } \xi\in \Gamma(M,\shP_{K})\Bigr\},
$$
and is studied in \cite{LW2, W1}.
In \cite{LW2}, Leung and the first author give a geometric description of the weight decomposition of $\shH_{K}$. In \cite{W1}, the first author shows that the dimension of $\shH_{K}$ coincides with the number of integral point of the Delzant polytope and constructs a canonical basis $\{ \delta^{m}_{K}\}$ of $\shH_{K}$,
where $\delta_{K}^{m}$ is defined by formula (\ref{dis-basis0}) for $t=0$.
\begin{theorem}\cite[Theorem 3.2]{W1}
Under assumption $(*)$, $\{\delta_{K}^{m}\}_{m\in P_{\ZZ}}$ forms a basis of quantum space $\shH_{K}$. In particular, $\mathrm{dim} \shH_{K}=\mathrm{dim} H^{0}(X,L)$.
\end{theorem}

\section{Main result}
Recall that we work under the following standing assumption $(*)$: let $(M, \omega, J)$ be a $2n$-dimensional toric variety determined by a Delzant polytope $P$, with associated moment map $\mu: M \twoheadrightarrow P$. Let $K \subset T^n$ be a $k$-dimensional subtorus and $\rho_K: K \to \mathrm{Diff}(M, \omega, J)$ be the subaction with corresponding moment map $\mu_K: M \to \mathfrak{t}_K^*$. Notice that the short exact sequence
\begin{center}
    \begin{tikzcd}
        K \ar[r] &  T^{n} \ar[r] & Q := T^{n}/K
    \end{tikzcd}
\end{center}
always splits. Once and for all, we fix a splitting $i_Q: Q \to T^n$.
Let $\phi: \mathfrak{t}_Q^* \to \mathbb{R}$ be a strictly convex smooth function. Denote by $X_\varphi$ the Hamiltonian vector field associated with $\varphi := \phi \circ \mu_Q$, where $\mu_Q = i_Q^* \circ \mu$ and $i_Q^*: \mathfrak{t}^* \to \mathfrak{t}_Q^*$ is the natural projection.

\subsection{Degeneration from mixed to real polarizations}

To study how geometric quantization depends on the choice of polarizations, it is natural to examine the relationship between different types of polarizations.
In this paper, we investigate the connection between the mixed polarization $\shP_{K}$
(see formula (\ref{def-mixp})), which is associated with the Hamiltonian subtorus $K$-action, and the real polarization $\shP_{\RR}$ (see formula (\ref{def-realp})) associated with the Hamiltonian $T^{n}$-action.
 
 In this subsection, we show that applying the imaginary-time flow generated by $X_\varphi$ to $\shP_K$ yields a one-parameter family of mixed polarizations $\shP_{K,t}$ on $(M, \omega)$ (see Theorem \ref{thm1}), and that $\shP_{K,t}$ converges to $\shP_{\RR}$ as $t \to \infty$ (see Theorem \ref{plim}).

\subsubsection{Construction of a one-parameter family of mixed polarizations}
We construct a one-parameter family mixed polarizations by applying the imaginary time flow of $X_\varphi$ on $\shP_K$.

We begin with considering a local model: a K\"ahler manifold $T_{\CC}^{n-l}\times\CC^{l}$ equipped with a Hamiltonian $T^{n}$-action by holomorphic isometries. We assume that its complex structure is determined by the standard holomorphic coordinates $(w_1, \cdots, w_{n-l}, z_1, \cdots, z_l)$, although its K\"ahler form is not assumed to be the standard one. We further assume that the action of $T^{n}=T^{p} \times T^{n-l-p}  \times T^{l}$ is given by the standard multiplication of $T^{p} \times T^{n-l-p}$ on $T_{\CC}^{n-l}$ together with a unitary representation of $T^{l}$ on $\CC^{l}$.

The standard basis $\{\xi_{1}, \cdots, \xi_{n}\}$ of the Lie algebra of $T^n$ 
induces global linear coordinates $(x_1, \dots, x_n)$ on its dual. 
With respect to these coordinates, we express the moment map $\mu$ for the Hamiltonian $T^n$-action in terms of its components:
$$\mu=(\mu_{1},\cdots, \mu_{n}):T^{n-l}_{\CC} \times \CC^{l} \rightarrow \mathbb{R}^n.$$

We now focus on the subtorus $Q = T^q \subset T^n$ consisting of the $(p+1)$st through $(p+q)$th circle factors, where $n-l-p\le q\le n-p$. Its subtorus action has moment map
$$\mu_Q = (\mu_{p+1}, \cdots, \mu_{p+q}): T^{n-l}_{\CC} \times \CC^{l} \rightarrow \mathbb{R}^q.$$
In other words, for each $\gamma = p+1, \cdots, p+q$, the fundamental vector field $\xi_{\gamma}^{\#}$ associated with $\xi_{\gamma}$ is the Hamiltonian vector field associated with the component function $\mu_{\gamma}$. According to our assumption,
\begin{equation}
    \label{eq: fundamental vector field 1}
    \xi_{\gamma}^{\#} = i \left( w_{\gamma} \frac{\partial}{\partial w_{\gamma}} - \overline{w}_{\gamma} \frac{\partial}{\partial \overline{w}_{\gamma
    }} \right) \quad \text{for } \gamma = p+1, \cdots, n-l;
\end{equation}
and there exists a matrix $(b_{\gamma k})$ over $\mathbb{Z}$, where the index $\gamma$ ranges over $n-l+1, \cdots, p+q$ and $k$ ranges over  $1, \cdots, l$, such that
\begin{equation}
    \label{eq: fundamental vector field 2}
    \xi^{\#}_{\gamma}=\sum_{k=1}^{l} i b_{\gamma k} \left( z_{k} \frac{\partial}{\partial z_{k}} - \overline{z}_{k} \frac{\partial}{\partial \overline{z}_{k}} \right) \quad \text{for } \gamma=n-l+1, \cdots, p+q.
\end{equation}

Let $\phi: \mathbb{R}^q \to \mathbb{R}$ be a strictly convex smooth function, and let $X_\varphi$ be the Hamiltonian vector field associated with $\varphi := \phi \circ \mu_Q$.

\begin{lemma}\label{local coordinate lem}
Given the action of $T^{n}= T^{p} \times T^{n-l-p}  \times T^{l}$ on $T^{n-l}_{\CC} \times \CC^{l}$ as described above, 
$e^{-itX_{\varphi}}$ can be applied to $w_{j}$ and $z_{k}$ completely, for $ j=1,\cdots,n-l, k=1, \cdots, l$. Moreover,
\begin{enumerate}
\item $w_{j}^{t}:=e^{-itX_{\varphi}}w_{j}=w_{j}$, for $j=1,\cdots, p$;
\item $w_{j}^{t}:=e^{-itX_{\varphi}}w_{j}=w_{j}e^{t\frac{\partial \varphi}{\partial \mu_{j}}}$, for $j=p+1, \cdots, n-l$;
\item $ z_{k}^{t}:=e^{-itX_{\varphi}}z_{k}=z_{k}e^{t \sum_{\gamma=n-l+1}^{p+q} \frac{\partial \varphi}{\partial \mu_{\gamma}} b_{\gamma k}}$, for $k=1,\cdots,l$.
\end{enumerate}
 \end{lemma} 
\begin{proof}  
  The Hamiltonian vector field associated with $\varphi = \phi \circ \mu_Q$ is given by
  $$X_{\varphi}=\sum_{\gamma=p+1}^{p+q} \frac{\partial \varphi}{\partial \mu_{\gamma}} \xi_{\gamma}^{\#}.$$
  We then apply equations \eqref{eq: fundamental vector field 1} and \eqref{eq: fundamental vector field 2} to each of the following cases:
\begin{enumerate}
    \item Fix $j=1,\cdots p$. This gives $X_{\varphi} (w_{j})=0$ and therefore $ w_{j}^{t}:=e^{-itX_{\varphi}}w_{j}=w_{j}$.
    \item Fix $j=p+1,\cdots n-l$. This gives $-iX_{\varphi} (w_{j})=w_{j} \frac{\partial \varphi}{\partial \mu_{j}}.$
 As a $Q$-invariant function, $\frac{\partial \varphi}{\partial \mu_{j}}$ is annihilated by $X_{\varphi}$, hence we obtain by induction that
$$(-i)^{r}X_{\varphi}^{r}(w_{j})=w_{j}\left( \frac{\partial \varphi}{\partial \mu_{j}} \right)^{r} \quad \text{for all } r \in \mathbb{N}.$$
Therefore $e^{-itX_{\varphi}}$ can be applied to $w_{j}$ completely and 
 $$w_{j}^{t}:=e^{-itX_{\varphi}}w_{j}=\sum_{r=0}^\infty \frac{1}{r!}(-itX_{\varphi})^{r}(w_{j})=w_{j}e^{t\frac{\partial \varphi}{\partial \mu_{j}}}.$$
 \item Fix $k=1,\cdots,l$. We obtain $-iX_{\varphi} z_{k}=z_{k} \sum_{\gamma=n-l+1}^{p+q}\frac{\partial \varphi}{\partial \mu_{\gamma}} b_{\gamma k}$. Again, by the $Q$-invariance, $\sum_{\gamma=n-l+1}^{p+q}  \frac{\partial \varphi}{\partial \mu_{\gamma}} b_{\gamma k}$ is annihilated by $X_{\varphi}$. By the same argument as in $(2)$, we prove that 
$e^{-itX_{\varphi}}$ can be applied to $z_{k}$ completely for $k=1,\cdots,l$, and
$$z_{k}^{t}:=e^{-itX_{\varphi}} z_{k}=z_{k}\sum_{r=0}^\infty \frac{1}{r!}\left( t \sum_{\gamma=n-l+1}^{p+q} \frac{\partial \varphi}{\partial \mu_{\gamma}} b_{\gamma k}\right)^{r}=z_{k}e^{t\left(\sum_{\gamma=n-l+1}^{p+q} \frac{\partial \varphi}{\partial \mu_{\gamma}} b_{\gamma k}\right)}.$$
\end{enumerate}
\end{proof}

\begin{remark}
    In the above lemma, for the ease of notations, we write $\tfrac{\partial \varphi}{\partial \mu_\gamma}$ in place of $\mu_Q^*\left(\tfrac{\partial \phi}{\partial x_{\gamma}} \right)$ for $\gamma = p+1, \cdots, p+q$, where we regard $(x_{p+1},\ldots,x_{p+q})$ as the standard coordinates on $\mathbb{R}^{q}$. 
\end{remark}

\begin{theorem}\label{thm1} Under assumption $(*)$, for any $t\ge0$,
$$\shP_{K,t}:=e^{-it\shL_{X_{\varphi}}}\shP_{K}$$
is a (singular) mixed polarization with $\rk_{\RR}(\shP_{K,t})=\rk_{\RR}(\shP_{K})$.
\end{theorem}

\begin{proof}
Fix any point $x \in M$. Denote the stabilizer of the $T^{n}$-action at $x$ by $H^{T}_{x}$, and its dimension by $l$. Let $p$ be the integer such that $H^{K}_{x} := H^{T}_{x}\cap K$, the stabilizer of $K$ at $x$, has dimension $k-p$. Define $T^{p}=K/H^{K}_{x}$ and $T^{n-l-p}=Q/H_{x}^{Q}$, where $H^{Q}_{x}=H^{T}_{x}/H^{K}_{x}$. We have $T^{n} \cong T^{p} \times T^{n-l-p} \times T^{l},$ where $T^{l}\cong  H^{Q}_{x} \times H^{K}_{x} $.

By the holomorphic slice theorem \cite[Theorem 1.12]{Sj} (see Theorem \ref{thm2-2}), 
there exists a $T^{n}$-equivariant map from a $T^n$-invariant open neighborhood of $x$ to a $T^n$-invariant open subset of 
$T^{n-l}_{\CC} \times \CC^{l}$, where $T^{n}= T^{p} \times T^{n-l-p} \times T^{l}$ acts on $T^{n-l}_{\CC} \times \CC^{l}$ by the standard multiplication of $ T^{p} \times T^{n-l-p} $ on $T^{n-l}_{\CC}$ together with a unitary representation of $T^{l}$ on $\CC^{l}$. 

Let $(w_{1},\cdots,w_{n-l}, z_{1}, \cdots,z_{l})$ be the standard holomorphic coordinates on $T^{n-l}_{\CC} \times \CC^{l}$. Since each $w_{j}$ is nowhere vanishing, we introduce logarithmic polar coordinates $(\rho_j,\theta_j)$ by $w_{j}=e^{\rho_j+i\theta_{j}}$.
By \cite[Theorem 3.1]{MNGW} (see also the proof of \cite[Theorem 3.21]{LW1}), $(\shP_{K})_{x}$ is given by
\begin{align*}(\shP_{K})_{x} &= \Span_{\CC} \left\{\frac{\partial }{\partial \theta_{1}} , \cdots, \frac{\partial }{\partial \theta_{p}}, \frac{\partial }{\partial \bar{w}_{p+1}},\cdots, \frac{\partial }{\partial \bar{w}_{n-l}},\frac{\partial }{\partial \bar{z}_{1}},\cdots, \frac{\partial }{\partial \bar{z}_{l}} \right\}\\
&=\Span_{\CC} \left\{\frac{\partial }{\partial \theta_{1}} , \cdots, \frac{\partial }{\partial \theta_{p}}, X_{{w}_{p+1}},\cdots, X_{{w}_{n-l}} ,X_{{z}_{1}},\cdots,X_{{z}_{l}}  \right\}.
\end{align*}

By Lemma~\ref{local coordinate lem}, the action of imaginary time flow $e^{-itX_{\varphi}}$ on the local coordinate $w_j$ is
$$w_{j}^{t}:=e^{-itX_{\varphi}}w_{j}=w_{j}e^{t\frac{\partial \varphi}{\partial \mu_{j}}}$$
for $j = p+1, \ldots, n-l$. Similarly, for $k = 1, \ldots, l$, we have
$$z_{k}^{t}:=e^{-itX_{\varphi}} z_{k}=z_{k}e^{t\left(\sum_{\gamma=n-l+1}^{p+q} \frac{\partial \varphi}{\partial \mu_{\gamma}} b_{\gamma k}\right)}.$$
Consequently, the induced action on the corresponding vector fields yields
\[
e^{-it\mathcal{L}_{X_{\varphi}}}X_{w_{j}} = X_{w_{j}^{t}}, \qquad e^{-it\mathcal{L}_{X_{\varphi}}}X_{z_{k}} = X_{z_{k}^{t}},
\]
for $j = p+1, \ldots, n-l$ and $k = 1, \ldots, l$.
Let $\shP_{t}$ be the K\"ahler polarization associated with the complex structure $J_{t}$. By the proof of \cite[Theorem 3.19]{LW1}, $(w_{1}^{t},\cdots,w_{n-l}^{t},z_{1}^{t},\cdots,z_{l}^{t})$ form a $J_{t}$-holomorphic coordinate around $x$.
Then we have:
\begin{align*}
(\shP_{K,t})_{x}&=(e^{-it\shL_{X_{\varphi}}}\shP_{K})_{x}\\
&= \Span_{\CC} \left\{\frac{\partial }{\partial \theta_{1}} , \cdots, \frac{\partial }{\partial \theta_{p}}, X_{{w}_{p+1}^{t}},\cdots, X_{{w}_{n-l}^{t}} ,X_{{z}_{1}^{t}},\cdots,X_{{z}_{l}^{t}}  \right\}\\
 &= \Span_{\CC} \left\{\frac{\partial }{\partial \theta_{1}} , \cdots, \frac{\partial }{\partial \theta_{p}}, \frac{\partial }{\partial \bar{w}_{p+1}^{t}},\cdots, \frac{\partial }{\partial \bar{w}_{n-l}^{t}},\frac{\partial }{\partial \bar{z}_{1}^{t}},\cdots, \frac{\partial }{\partial \bar{z}_{l}^{t}} \right\}\\
 &=\left(\left((\ker d\mu_{K} \otimes \CC)\cap \shP_{t} \right) \oplus (\Im d\rho_{K}\otimes \CC)\right)_{x}
\end{align*}
This implies that $\shP_{K,t}=\left((\ker d\mu_{K} \otimes \CC)\cap \shP_{t} \right) \oplus (\Im d\rho_{K}\otimes \CC)$.
Thus, by \cite[Theorem 1.1]{LW1}, 
$\shP_{K,t}$ is a (singular) mixed polarization on $X$.
\end{proof}

\subsubsection{Degeneration of the mixed polarizations}
Now, we show (see Theorem \ref{plim}) that the one-parameter family of mixed polarizations $\shP_{K,t}$ interpolates between the mixed polarization $\shP_{K}$ and the real polarization $\shP_{\RR}$, i.e. $\shP_{K,0}=\shP_{K}$ and $\lim_{t\rightarrow \infty} \shP_{K,t}=\shP_{\RR}$. 

We begin with the following lemma, which will be used in the proof of Theorem \ref{plim}.

\begin{lemma}\label{lem-com}
    Let $V$ be a finite-dimensional real vector space and let $W \subset V$ be a linear subspace. Let $\{P_t\}_{t \in \RR}$ be a continuous family of subspaces of $V$, all of the same dimension, such that the limit $P_\infty := \lim_{t \rightarrow \infty} P_t$ exists and $\dim P_\infty = \dim P_t$ for all $t \in \mathbb{R}$. Further suppose that $\dim (W \cap P_\infty) = \dim (W \cap P_t)$ for all $t \in \mathbb{R}$. Then
    $$\lim_{t \rightarrow \infty} (W \cap P_t) = W \cap P_\infty \quad \text{and} \quad \lim_{t \rightarrow \infty} (P_t + W) = P_\infty + W.$$
\end{lemma}
\begin{proof}
    Let $p$ denote the common dimension of $P_t$ and of $P_\infty$, and let $\ell$ denotes the common dimension of $W \cap P_t$ and of $W \cap P_\infty$. Consider the Schubert-type locus
	\[
	\Omega = \{\, P \in \mathrm{Gr}_p(V) \mid \dim(P \cap W) = \ell \,\},
	\]
    where $\mathrm{Gr}_p(V)$ denotes the Grassmannian of $p$-planes in $V$. By assumption, $P_t \in \Omega$ for all $t$ and $P_\infty \in \Omega$. It therefore suffices to prove that the maps
	\begin{align*}
		\pi: \Omega \to \mathrm{Gr}_{\ell}(W), & \quad P \mapsto P \cap W,\\
		\Phi: \Omega \to \mathrm{Gr}_{p + a - \ell}(V), & \quad P \mapsto P + W,
	\end{align*}
    are continuous at $P_\infty \in \Omega$, where $a$ denotes the dimension of $W$.\par
	Choose a subspace $\widetilde{W} \subset W$ complementary to $P_\infty \cap W$ and a subspace $L \subset V$ complementary to $P_\infty + W$. Then, $V = P_\infty \oplus \widetilde{W} \oplus L$. Let
    $$U_{P_\infty} := \{ P \in \mathrm{Gr}_p(V): P \cap (\widetilde{W} \oplus L) = \{0\} \}$$
    be the standard affine neighbourhood of $P_\infty$ in $\mathrm{Gr}_p(V)$. Via the graph parametrization, $U_{P_\infty}$ is identified with $\Hom(P_\infty,\widetilde{W}\oplus L)$ by $\varphi \mapsto \operatorname{Graph}(\varphi) = \{ v + \varphi(v): v \in P_\infty\}$. Writing $\varphi=\varphi_{\widetilde{W}}+\varphi_L$ according to the decomposition $\widetilde{W }\oplus L$, one checks that
    \begin{equation*}
        \Omega \cap U_{P_\infty} \cong \{ \varphi \in \operatorname{Hom}(P_\infty, \widetilde{W} \oplus L): \varphi_L \vert_{P_\infty \cap W} = 0 \}.
    \end{equation*}

    Similarly, the neighbourhood
	\begin{equation*}
		U_{P_\infty \cap W} := \{ P \in \mathrm{Gr}_{\ell}(W): P \cap \widetilde{W} = \{0\} \}
	\end{equation*}
	of $P_\infty \cap W$ in $\mathrm{Gr}_{\ell}(W)$ is identified with $\operatorname{Hom}(P_\infty \cap W, \widetilde{W})$, and the neighbourhood
	\begin{equation*}
		U_{P_\infty + W} := \{ P \in \mathrm{Gr}_{p + a - \ell}(V): P \cap L = \{0\} \}
	\end{equation*}
	of $P_\infty + W$ in $\mathrm{Gr}_{p + a - \ell}(V)$ is identified with $\operatorname{Hom}(P_\infty \oplus \widetilde{W}, L)$. Under these identifications, for $\varphi\in\Omega\cap U_{P_\infty}$ we obtain
	$$\pi(\varphi) = \varphi_{\widetilde{W}} \vert_{P_\infty \cap W} \quad \text{and} \quad \Phi(\varphi)(v + w) = \varphi_L(v) \quad \text{for all } v \in P_\infty \text{ and } w \in \widetilde{W}.$$
    In particular, both $\pi$ and $\Phi$ are given by linear projections in these affine coordinates and are therefore continuous at $P_\infty$. This completes the proof.
\end{proof}

\begin{theorem}\label{plim}
Under assumption $(*)$, we have:
$$\lim_{t \rightarrow \infty} \shP_{K,t}=\shP_{\RR}.$$
\end{theorem}

\begin{proof}
By the proof of Theorem \ref{thm1}, we know that
\begin{align*}
   \shP_{K,t}=\left((\ker d\mu_{K} \otimes \CC)\cap \shP_{t} \right) \oplus (\Im d\rho_{K}\otimes \CC),
\end{align*}
where $\shP_{t}$ denotes the K\"ahler polarization obtained from the imaginary-time flow $e^{-itX_\varphi}$. We first show that
\begin{equation}\label{eq: desired P_R}
    \shP_{\RR} = \left( (\ker d\mu_{K} \otimes \CC) \cap \left(\lim_{t \rightarrow \infty} \shP_{t}\right) \right) \oplus (\Im d \rho_{K} \otimes \CC).
\end{equation}
Recall that $\shP_{\RR} = \left((\ker d\mu\otimes \CC) \cap \shP_{J}\right) \oplus (\Im d \rho \otimes \CC)$. Since $\Im d \rho = \Im d \rho_{Q} \oplus \Im d \rho_{K}$,
\begin{equation}\label{eq:limit-claim}
    \shP_{\RR} = \left((\ker d\mu\otimes \CC) \cap \shP_{J}\right) \oplus (\Im d \rho_{Q} \otimes \CC) \oplus (\Im d \rho_{K} \otimes \CC).
\end{equation}
By \cite[Theorem 3.21]{LW1}, we have
\begin{equation*}
\lim_{t \rightarrow \infty} \shP_{t}=\left((\ker d \mu_{Q}\otimes \CC) \cap \shP_{J}\right)\oplus (\Im d \rho_{Q} \otimes \CC).
\end{equation*}
Since $\Im d \rho_{Q} \subset \ker d\mu_{K}$, 
it follows that
\begin{align*}
    (\ker d\mu_{K} \otimes \CC) \cap \left(\lim_{t \rightarrow \infty} \shP_{t}\right) = & \left((\ker d\mu_{K} \otimes \CC) \cap (\ker d\mu_{Q} \otimes \CC) \cap \shP_{J}\right) \oplus (\Im d \rho_{Q} \otimes \CC)\\
    = & \left((\ker d\mu\otimes \CC) \cap \shP_{J}\right) \oplus (\Im d \rho_{Q} \otimes \CC).
\end{align*}
Substituting this into \eqref{eq:limit-claim} yields the desired decomposition of $\shP_{\RR}$ in \eqref{eq: desired P_R}.

Finally, from the coordinate frame descriptions of $(\shP_{K,t})_{x}$ in the proof of Theorem~\ref{thm1} and of $(\shP_{\RR})_{x}$ in \cite[Theorem 3.21]{LW1}, we see that for every $x \in M$ both fibers have dimension $n$. By Lemma~\ref{lem-com}, intersections with the fibxed subspace $(\ker d\mu_{K} \otimes \CC)_{x}$ commute with taking limits. Also, direct sums with the fixed subspace $(\Im d \rho_{K} \otimes \CC)_{x}$ are preserved under limits. Therefore,
$$\lim_{t \rightarrow \infty} \shP_{K,t} = \left( (\ker d\mu_{K} \otimes \CC) \cap \left(\lim_{t \rightarrow \infty} \shP_{t}\right) \right) \oplus (\Im d \rho_{K} \otimes \CC) = \shP_{\RR}.$$
This completes the proof.
\end{proof}
\subsection{Convergence of quantum Hilbert spaces}
In this section, we show that the one-parameter family of quantum Hilbert
spaces associated with the mixed polarizations converges to the quantum Hilbert space for the real polarization.

Let $L$ be the $T^n$-equivariant prequantum line bundle of $(M, \omega)$ determined by the polytope $P$. Consider the natural embedding of smooth sections of $L$ into distributional
sections of $L$, defined using the Liouville measure: for any open set
$U\subset M$, 
\begin{align*}
i: \Gamma(U, L) & \rightarrow \Gamma_{c}(U,L^{-1})'\\
s & \mapsto i(s)(\phi)= \int_{U} \langle s, \phi \rangle \cdot \frac{1}{n!} \omega^n,
\end{align*}
where $\phi\in \Gamma_c(U,L^{-1})$ denotes a test section.

Following \cite{BFMN,W1}, we extend the operator $\nabla_{\xi}$ from smooth
sections to distributional sections, and we denote this extension by the same
symbol $\nabla_{\xi}$. 
For a general distributional section $\delta \in \Gamma_{c}(M,L^{-1})'$, 
the distribution $\nabla_{\xi}\delta$ is
defined by duality. Namely,
\begin{equation}\label{eqcon}
(\nabla_{\xi}\delta)(\tau)=\delta\!\left({}^{t}\nabla_{\xi}\tau\right),
\end{equation}
for every test section $\tau\in \Gamma_{c}(M,L^{-1})$. Here ${}^{t}\nabla_{\xi}$
denotes the transpose of $\nabla_{\xi}$ with respect to the Liouville measure, given by
\begin{equation}\label{trans-conn}
{}^{t}\nabla_{\xi}\tau
=
-\bigl(\operatorname{div}(\xi)\,\phi+\nabla_{\xi}^{-1}\tau\bigr),
\end{equation}
where $\operatorname{div}(\xi)$ denotes the divergence of $\xi$ with respect to the Liouville measure.

Given a polarization $\shP$, we define the associated quantum Hilbert space by
\begin{equation}
\shH
=
\Bigl\{\delta\in \Gamma_{c}(M,L^{-1})' \ \Big|\ 
\nabla_{\xi}\delta=0 \text{ for all } \xi\in \Gamma(M,\shP)\Bigr\}.
\end{equation}

In particular, for the real polarization
$\shP_{\RR}$ defined by equation (\ref{def-realp}) on a toric variety $M$, by \cite[Proposition 3.1]{BFMN}, the associated quantum Hilbert space $\shH_{\RR}$
admits a canonical basis
$\{\delta^{m}\}_{m\in P\cap \mathfrak{t}_{\mathbb{Z}}^*}$, indexed by integral points in the moment polytope $P$. The element $\delta^{m}$ is
the distributional section supported on the Bohr--Sommerfeld fiber
$\mu^{-1}(m)$, characterized as follows: for any $T^{n}$-invariant open subset
$W\subset M$ containing $\mu^{-1}(m)$,
\begin{equation}\label{eqn:delta^m(tau)}
\delta^{m}(\tau)
=
\int_{\mu^{-1}(m)} e^{i\langle \ell(m),\theta\rangle}\, \tau_v \vol_{\mu^{-1}(m)},
\qquad
\forall\, \tau\in \Gamma_c(W,L^{-1}).
\end{equation}
Here, $\vol_{\mu^{-1}(m)}$ is a $T^{n}$-invariant volume form on the torus
$\mu^{-1}(m)$, to be defined in \eqref{fibre-volume}. The map $\ell(x)=Ax-\lambda$ is the affine change of variables,
where $A=((\nu_i)_j)$ and $\nu_j,\lambda_j$ are as in \eqref{eqn:l_r}. The function $\tau_v$ is determined as follows. Choose a vertex $v\in P$
such that $m$ lies in the union $\mathring{P}_v$ of the relative interiors of all
faces whose closure contains $v$. 
Trivialize $L^{-1}$ over
$V_v:=\mu^{-1}(\mathring{P}_v)$ by a local frame $1_v$. Then on $V_v$, write $\tau=\tau_v\cdot 1_v$.

Let $\{\sigma^{m}\}_{m\in P\cap \mathfrak{t}_{\ZZ}^{*}}$ be a basis of quantum space $\shH$ associated with $J$. For the mixed polarization $\shP_{K}$ defined by formula (\ref{def-mixp}),
by \cite[Theorem~3.2]{W1}, the associated quantum Hilbert space $\shH_{K}$
admits a canonical basis $\{\delta_{K}^{m}\}_{m\in P\cap \mathfrak{t}_{\ZZ}^{*}}$,
where $\delta_{K}^{m}$ is defined by
\begin{equation}\label{dis-basis0}
\delta_{K}^{m}(\tau)
:=
\int_{\check{M}^{m}} \bigl\langle \sigma^{m},\tau\bigr\rangle\,
\vol_{K}^{m}
\end{equation}
for every test section $\tau\in \Gamma_c(M,L^{-1})$, where 
$\vol_K^{m}$ is the volume form given in Proposition \ref{decomp-vol}.

Recall from our standing assumption (*) that $\phi\colon \mathfrak{t}_{Q}^{*}\to \RR$ is a strictly convex function, and $X_{\varphi}$ is the
Hamiltonian vector field of the smooth function $\varphi=\phi\circ \mu_{Q}$. For each $t > 0$, let $\shP_{t}$ denote the K\"ahler polarization associated with the complex structure $J_t$ defined by the imaginary-time flow $e^{-itX_{\varphi}}$ of $J$. 
Let $\shP_{K,t}$ be the polarization given in Theorem~\ref{thm1}. From the
proof of the same theorem we have
\begin{equation}\label{mixp-fam}
\shP_{K,t}
=
\bigl((\ker \d\mu_{K}\otimes \CC)\cap \shP_{t}\bigr)\oplus
(\Im \d\rho_{K}\otimes \CC).
\end{equation}
In order to study the relationship between the associated quantum spaces $\shH_{K,t}$, 
we extend the Kostant-Souriau operator $\hat{\varphi}$ associated with $\varphi=\phi\circ \mu_{Q}$ to distributional sections using the same formula and notation. 

Namely,
\begin{equation}\label{eq2-0-9}
\hat{\varphi}\colon \Gamma_{c}(M,L^{-1})'\to \Gamma_{c}(M,L^{-1})',\qquad
\delta\mapsto \hat{\varphi}\delta:=-i\nabla_{X_{\varphi}}\delta+\varphi \delta.
\end{equation}

\begin{definition}\label{def-imaginary-flow}
For $\delta \in \Gamma_{c}(M,L^{-1})'$, if the Lie series
$ 
\sum_{j=0}^{\infty} \frac{t^{j}}{j!}\hat{\varphi}^{j}\delta
$ 
converges absolutely and uniformly on compact subsets of $M \times \mathbb{R}$ in the sense that, for every $\tau \in \Gamma_{c}(M,L^{-1})$, the series
$ 
\sum_{j=0}^{\infty} \frac{t^{j}}{j!}(\hat{\varphi}^{j}\delta)(\tau)
$ 
does so, then we denote its sum by $e^{t\hat{\varphi}}\delta$ and say that $e^{t\hat{\varphi}}$ can be applied to $\delta$.
\end{definition}

\begin{theorem}\label{imaginary-flow}
Under assumption $(*)$, for any $t >0$,  $e^{t\hat{\varphi}} : \shH_{K,0} \rightarrow \shH_{K,t}$ is a $T^{n}$-equivariant linear isomorphism.
\end{theorem}
\begin{proof}
Note that, the polarization $\shP_{K,0}$ coincides with mixed polarization $\shP_{K}$ defined by  formula \ref{def-mixp}.
Let $\{\sigma^{m}\}_{m\in P\cap \mathfrak{t}_{\ZZ}^{*}}$ be a basis of quantum space $\shH$ associated with $J$. 
By \cite[Theorem~3.2]{W1}, the associated quantum Hilbert space $\shH_{K}$
admits a canonical basis $\{\delta_{K}^{m}\}_{m\in P\cap \mathfrak{t}_{\ZZ}^{*}}$, 
where $\delta_{K}^{m}$ is defined by 
\begin{equation}
\delta_{K}^{m}(\tau)
:=
\int_{\check{M}^{m}} \bigl\langle \sigma^{m},\tau\bigr\rangle\,
\vol_{K}^{m}
\end{equation}
for every $\tau\in \Gamma_c(M,L^{-1})$. For any $m\in P\cap \mathfrak{t}_{\ZZ}^{*}$ and every test section $\tau\in \Gamma_c(M,L^{-1})$, by equation (\ref{trans-conn}),
\begin{align*}
(\nabla_{X_{\varphi}}\delta_{K}^{m})(\tau)=\int_{\check{M}^{m}} \bigl\langle \sigma^{m},-\nabla_{X_{\varphi}}^{-1}\tau\bigr\rangle\,
\vol_{K}^{m}.
\end{align*}
This implies
\begin{equation}
(\hat{\varphi}\delta_{K}^{m})(\tau)=\int_{\check{M}^{m}} \bigl\langle \hat{\varphi}\sigma^{m},\tau\bigr\rangle\,
\vol_{K}^{m}.
\end{equation}
and  
\begin{align*}
\sum_{j=0}^{\infty} \frac{t^{j}}{j!}(\hat{\varphi}^{j}\delta_{K}^{m})(\tau)= \sum_{j=0}^{\infty} \int_{\check{M}^{m}} \bigl\langle \frac{t^{j}}{j!}\hat{\varphi}^{j}\sigma^{m})(\tau)\bigr\rangle \vol_{K}^{m}=\int_{\check{M}^{m}} \bigl\langle e^{t \hat{\varphi}}\sigma^{m},\tau\bigr\rangle\,
\vol_{K}^{m}.
\end{align*} 
By \cite[Theorem~3.7]{LW3}, $e^{t\hat{\varphi}}\sigma^{m} \in \shH_{t}$, and the family
\begin{equation}\label{eqn:sigma_t^m}
\{\sigma_{t}^{m}:=e^{t\hat{\varphi}}\sigma^{m}\}_{m\in P\cap \mathfrak{t}_{\ZZ}^{*}}
\end{equation}
form a basis of the quantum Hilbert space $\shH_{t}$ associated with
$\shP_{t}$.
We conclude that $e^{t\hat{\varphi}}$ can be applied to $\delta_{K}^{m}$ and 
\begin{equation}\label{dis-ba}
(e^{t\varphi}\delta^{m}_{K})(\tau)=\int_{\check{M}^{m}} \bigl\langle \sigma^{m}_{t},\tau\bigr\rangle\,
\vol_{K}^{m}.
\end{equation}
By \eqref{mixp-fam}, \eqref{dis-ba}, and \cite[Theorem~3.2]{W1}, it follows that
$e^{t\hat{\varphi}}\delta_{K}^{m}$ lies in the quantum space $\shH_{K,t}$ associated with the polarization $\shP_{K,t}$, and that
$\{e^{t\hat{\varphi}}\delta_{K}^{m}\}_{m\in P\cap \mathfrak{t}_{\ZZ}^{*}}$
forms a basis of $\shH_{K,t}$. The proof of the $T^{n}$-equivariance of $e^{t\hat{\varphi}}$ is standard and follows the same argument as in \cite[Theorem~3.7]{LW3}.
\end{proof}

For the purpose of proving our last main result (Theorem \ref{thm-hlim}), we will need to integrate along the fibres of the moment map $\mu: M \to \fot^{*}$. Since $\mu$ is a smooth submersion only over the open dense subset $\mathring{M} \subset M$ (cf. \cite[Section 3]{DH}), we first clarify the meaning of integration along its fibres.

Recall that if $p: E \to B$ is a fibre bundle over an oriented smooth manifold $B$ with compact oriented fibres, then integration along the fibres defines a linear map
$$p_*: \Omega^*(E) \to \Omega^*(B).$$
It satisfies the projection formula: for any differential forms $\alpha \in \Omega^*(E)$ and $\beta \in \Omega^*(B)$,
\begin{equation}
    \label{projection formula}
    p_*(\alpha \wedge p^*\beta) = (p_*\alpha) \wedge \beta.
\end{equation}
The same construction extends to continuous differential forms, producing continuous differential forms on the base.

We now fix an orientation of $T^{n}$ and let $\vol_{T^{n}}$ denote the normalized $T^{n}$-invariant volume form on $T^{n}$, satisfying $\int_{T^{n}} \vol_{T^{n}} = 1$. For a point $x \in P$, let $H_x\subset T^n$ denote the stabilizer of the fibre $\mu^{-1}(x)$. With the induced orientation, $H_x$ is an oriented submanifold of $T^{n}$. The fibre $\mu^{-1}(x) \cong T^{n} / H_x$ is a torus. The quotient map $q_x: T^{n} \to T^{n} / H_x \cong \mu^{-1}(x)$ then defines a fibre integration map, and we set
\begin{equation}\label{fibre-volume}
    \vol_{\mu^{-1}(x)} := (q_x)_*(\vol_{T^{n}}).
\end{equation}
This gives a $T^{n}$-invariant volume form on the fibre $\mu^{-1}(x)$. Notice that, for any $m \in P \cap \mathfrak{t}_{\ZZ}^{*}$ and any weight-$m$ section $\sigma \in \Gamma(M, L)$, since the pointwise norm $|\sigma|$ is $T^n$-invariant, we may regard it as a function on $P$.

\begin{lemma}\label{lem-fib-int}
    Let $f$ be a complex-valued smooth function on $M$. Then the function
    \begin{equation}
        \label{eqn-fib-int}
        \int_{\mu} f:
        P \to \mathbb{C}, \quad x \mapsto \int_{\mu^{-1}(x)} f_x \vol_{\mu^{-1}(x)},
    \end{equation}
    where $f_x := f \vert_{\mu^{-1}(x)}$, is continuous on $P$. We call $\int_{\mu} f$ the \emph{fibre integration of} $f$ \emph{along} $\mu$.
    
    Moreover, if $m \in P \cap \fot^{*}_{\ZZ}$ and $\sigma \in \Gamma(M, L)$ has weight $m$, then there exists a complex constant $C_\sigma$ of modulus $1$ such that for all $\tau \in \Gamma(M, L^{-1})$,
    \begin{equation}
    \label{eqn-fib-int-delta}
        \left( \int_{\mu} \langle \sigma, \tau \rangle \right)(m) = C_\sigma \cdot \lvert \sigma \rvert (m) \cdot \delta^{m}(\tau),
    \end{equation}
    where $\lvert \sigma \rvert$ denotes the pointwise norm of $\sigma$ and $\delta^{m}$ is defined in \eqref{eqn:delta^m(tau)}.
\end{lemma}
\begin{proof}
    Let $\operatorname{Bl}_M: P \times T^{n} \to M$ denote the continuous map defined in \cite[Subsection 2.3]{LY2} such that the following diagram commutes:
    \begin{center}
        \begin{tikzcd}
            P \times T^{n} \ar[r, "\operatorname{Bl}_M"] \ar[rd, "\operatorname{pr}_P"'] & M \ar[d, "\mu"]\\
            & P
        \end{tikzcd}
    \end{center}
    where $\operatorname{pr}_P: P \times T^{n} \to P$ denotes the canonical projection. 

    We pull $\vol_{T^{n}}$ back along the canonical projection $P \times T^{n} \to T^{n}$, denoted by the same symbol. Since $\operatorname{Bl}_M^* f$ is a continuous function on $P \times T^{n}$, the fibre integration
    $$(\operatorname{pr}_P)_*((\operatorname{Bl}_M^* f) \cdot \vol_{T^{n}})$$
    is a continuous function on $P$. For all $x \in P$, by equation \eqref{projection formula} we compute
    \begin{align*}
        \left( \int_\mu f \right)(x) = & \int_{\mu^{-1}(x)} (q_x)_* ( (q_x^* f_x) \vol_{T^{n}} )\\
        = & \int_{T^{n}} (q_x^* f_x) \vol_{T^{n}}\\
        = & (\operatorname{pr}_P)_*((\operatorname{Bl}_M^* f) \cdot \vol_{T^{n}}) (x).
    \end{align*}
    Thus, $\int_{\mu} f$ is a continuous function on $P$.
    
    Now, suppose that $\sigma \in \Gamma(M, L)$ has weight $m \in P \cap \fot^{*}_{\ZZ}$. Then $\sigma$ is locally expressed as
    $$\sigma = C_{\sigma} \cdot \lvert \sigma \rvert \cdot e^{i \langle \ell(m), \theta \rangle} \cdot 1_v$$
    for some $C_\sigma \in \mathbb{C}$ satisfying $| C_\sigma | = 1$. By \eqref{eqn:delta^m(tau)}, for all $\tau \in \Gamma(M, L^{-1})$, \eqref{eqn-fib-int-delta} holds.
\end{proof}

For the K\"ahler polarization $\shP_{J}$ on $M$, the associated quantum Hilbert space $\shH_{J}$ also admits a canonical basis $\{\sigma^{m}\}_{m\in P\cap \mathfrak{t}_{\ZZ}^{*}}$, where for each $m \in P\cap \mathfrak{t}_{\ZZ}^{*}$, $\sigma^{m}$ is a $J$-holomorphic section of weight $m$. According to Lemma \ref{lem-fib-int}, we rescale each $\sigma^{m}$ so that it satisfies $C_{\sigma^{m}} = 1$, that is
\begin{equation}
    \label{rescale}
    \left( \int_{\mu} \langle \sigma^{m}, \tau \rangle \right)(m) = \lvert \sigma^{m} \rvert (m) \cdot \delta^{m}(\tau) \quad \text{for all } \tau \in \Gamma(M, L^{-1}).
\end{equation}

Our next goal is to clarify the relationship between the quantum spaces
$\shH_{K,t}$ associated with the mixed polarizations and $\shH_{\RR}$ associated
with the real polarization, via the one-parameter family
$\{\shP_{K,t}\}_{t\ge 0}$.

We begin by recalling the mixed-polarization setup. For any $m\in P$, set
$$
M_{K}^{m}:=\mu_{K}^{-1}\!\bigl(i_{K}^{*}(m)\bigr)\subset M,
$$
where $i_{K}^{*}\colon \mathfrak{t}^{*}\to \mathfrak{t}_{K}^{*}$ is the dual map
induced by the inclusion $i_{K}\colon \mathfrak{t}_{K}\hookrightarrow
\mathfrak{t}$, and $\mu_{K}\colon M\to \mathfrak{t}_{K}^{*}$ is the corresponding
moment map. The submanifold $M_{K}^{m}$ is $K$-invariant. Since $K$ is fixed throughout this paper, we henceforth write $M^{m}$ in place of $M_{K}^{m}$. Let
$$
\pi\colon M^{m}\longrightarrow M^{m}/K
$$
denote the quotient map.

If $i_{K}^{*}(m)\in \mathfrak{t}_{K}^{*}$ is a regular value of $\mu_{K}$ and the
$K$-action on $M^{m}$ is free, then $M^{m}$ is a principal $K$-bundle over the quotient space $M^{m}/K$. In this case there is a unique symplectic form
$\omega_{m,K}$ on $M^{m}/K$ characterized by
$$
\pi^{*}\omega_{m,K}=\omega|_{M^{m}}.
$$
Following \cite{LW2,LW3}, we equip $M^{m}$ with a volume form
$$
\vol_{K}^{m}
:=
\pi^{*}\!\left(\vol_{\omega_{m, K}}\right)\wedge \iota_{\mathrm{v}}\alpha^{k},
$$
where $\vol_{\omega_{m, K}} := \tfrac{1}{(n-k)!}\,\omega_{m,K}^{\,n-k}$ is the Liouville volume form on $(M^{m}/K, \omega_{m, K})$, $\mathrm{v}$ is a fixed $K$-invariant volume
form on $\fot_{K}$ and $\alpha\in \Omega^{1}(M^{m},\mathfrak{t}_{K})$ is any principal $K$-connection. Also, $\alpha^k \in \Omega^k(M^{m}, \bigwedge^k \mathfrak{t}_{K})$ denotes the $k$-fold exterior product of $\alpha$, and $\iota_{\mathrm{v}}$ denotes contraction with $\mathrm{v} \in \bigwedge^k \mathfrak{t}_{K}^*$.

\begin{remark}
\label{Remark: abbreviated notation}
This expression is independent of the choice of principal connection $\alpha$. Accordingly, we shall often use the abbreviated notation
$$
\vol_{K}^{m}
=
\pi^{*}\!\left(\vol_{\omega_{m, K}}\right)\wedge \mathrm{v}.
$$
The same convention will be adopted for analogous volume forms on principal bundles throughout the paper.
\end{remark}

For a general $m\in P$, \cite[Theorem~5.9]{SL} asserts that there exists a closed subgroup $H_K$ of $K$ and an open dense subset $\check{M}^{m} \subset M^{m}$ which is a principal $K/H_K$-bundle. Moreover, $\check{M}^{m}/K$ constitutes the largest open dense stratum of the symplectic reduction of $M$ at the level $\iota_K^*(m)$. 
Similar to the regular-value case, we denote by $\omega_{m,K}$ the symplectic form on this reduced space.

We further construct another open dense subset $\mathring{M}^{m}$ of $M^{m}$ that is better suited for the proof of our main result. To begin, consider the image
$$P^{m} := \mu(M^{m}) = \{ x \in P: \iota_{K}^*(x) = \iota_K^*(m) \} \subset \mathfrak{t}^*.$$
This is a polytope, being the intersection of the polytope $P$ with an affine subspace of $\fot^*$. Since $\mu_{K} = \iota_{K}^* \circ \mu$, the level set $M^{m}$ can be expressed as $M^{m} = \mu^{-1} ( P^{m} )$, which shows that $M^{m}$ is $T^{n}$-invariant. Let $\mathring{P}^{m}$ be the relative interior of $P^{m}$, and define
$$\mathring{M}^m = \mu^{-1}(\mathring{P}^{m}).$$
The fact that $\mathring{M}^{m}$ is an open dense subset of $\check{M}^{m}$ is a direct consequence of Lemma \ref{prin-bdl} in Appendix \ref{appendix}; it states that $\mathring{M}$ consists exactly of those points in $M^{m}$ whose $T^{n}$-stabilizer is minimal and $\mu \vert_{\mathring{M}^{m}}: \mathring{M}^{m} \to \mathring{P}^{m}$ is a principal torus bundle.

\begin{remark}
    The linear map $\iota_{Q}^*$ induces a bijection $P^{m} \to \iota_{Q}^*(P^{m})$, so $\iota_{Q}^*(P^{m})$ is a polytope in $\fot_{Q}^*$.
    Note that $\iota_{Q}^*(P^{m})$ may have dimension strictly less than that of $\fot_Q^*$.
\end{remark}

Define the quotient tori $\widetilde{K} := K / H_K$ and $\widetilde{T} := T^{n} / H$. Also define
\[
    \vol_{\widetilde{T}} := q_*(\vol_T),
\]
where $q: T \to \widetilde{T}$ is the quotient map. This gives a $T^{n}$-invariant volume form on $\widetilde{T}$. Let $\vol_{\omega_{m, K}}$ denote the Liouville volume form on $(\check{M}^m/K, \omega_{m, K} \vert_{\check{M}^m/K})$ and let $\vol_{P^{m}}$ be an Euclidean volume form on $P^{m}$. The following proposition provides a decomposition of $\vol_{K}^{m} |_{\mathring{M}^{m}}$, where $\vol_{K}^{m}$ is a volume form on $\check{M}^{m}$ constructed analogously to the regular-value case.

\begin{proposition}\label{decomp-vol}
    Equip the open subset $\mathring{M}^{m} / K \subset \check{M}^{m} / K$ with the restriction of the symplectic form $\omega_{m, K}$. 
    Then the principal $\widetilde{T} / \widetilde{K}$-bundle $\widetilde{\mu}: \mathring{M}^{m} / K \to \mathring{P}^{m}$ defined by the commutative diagram
        \begin{center}
            \begin{tikzcd}
                \mathring{M}^{m} \ar[d, "\pi \vert_{\mathring{M}^{m}}"'] \ar[rd, "\mu \vert_{\mathring{M}^{m}}"]\\
                \mathring{M}^{m} / K \ar[r, "\widetilde{\mu}"'] & \mathring{P}^{m}
            \end{tikzcd}
        \end{center}
        is a Lagrangian torus bundle. In particular, there exists a $\widetilde{K}$-invariant volume form $\mathrm{v}$ on $\fot_{\widetilde{K}}$, which defines a volume form on $\check{M}^{m}$ (see Remark \ref{Remark: abbreviated notation}):
        \[
            \vol_K^{m} := \pi^{*}\!\left(\vol_{\omega_{m, K}}\right)\wedge \mathrm{v},
        \]
        such that on $\mathring{M}^{m}$ this volume form decomposes as
        $$\vol_{K}^{m} |_{\mathring{M}^{m}} = \vol_{\widetilde{T}} \wedge \left( \mu \vert_{\mathring{M}^m} \right)^* \left( \vol_{P^{m}} \right).$$
\end{proposition}
\begin{proof}
    By the reduction-in-stages construction of Section 4 in \cite{SL}, 
    together with the fact that $\mathring{M}^{m}$ is a $T^{n}$-invariant open subset of $\check{M}^{m}$, the symplectic quotient $\mathring{M}^{m} / K$ carries a residual Hamiltonian $Q$-action. With respect to this action, the map
    $$\widetilde{\mu}: \mathring{M}^{m} / K \to \mathring{P}^{m}$$
    is a moment map. The $Q$-orbits in $\mathring{M}^{m} / K$ coincide with the fibres of $\widetilde{\mu}$. Comparing the rank of $\widetilde{\mu}$ with the dimension of the $Q$-orbits shows that these fibres are isotropic of maximal possible dimension; hence they are Lagrangian tori. 
    
    Since the base $\mathring{P}^{m}$ is contractible, the fibration $\widetilde{\mu}$ admits global action--angle coordinates. It follows that the Liouville volume form on $\mathring{M}^{m} / K$ decomposes as
    $$\vol_{\omega_{m, K}} |_{\mathring{M}^{m} / K} = \widetilde{\mu}^*\vol_{P^{m}} \wedge \mathrm{v}'$$
    for some $\widetilde{T}/\widetilde{K}$-invariant volume form $\mathrm{v}'$ on $\widetilde{T}/\widetilde{K}$. Choose a $\widetilde{K}$-invariant volume form $\mathrm{v}$ on $\mathfrak{t}_{\widetilde{K}}$ so that $\mathrm{v}' \wedge \mathrm{v} = (-1)^{m\widetilde{n}} \vol_{\widetilde{T}}$, where $\widetilde{n}$ is the dimension of $\widetilde{T}$. Pulling the above identity back via $\pi$ gives the desired decomposition.
\end{proof}

The next lemma gives an explicit relation between the basis elements
$\sigma_{t}^{m}$ defined in \eqref{eqn:sigma_t^m} and $\sigma^{m}$. 
Fix $m \in P$. To analyze the flow $e^{t\hat{\varphi}}$ acting on a $J$-holomorphic weight-$m$ section of $L$, we introduce the real-valued function on $\fot_{Q}^*$:
\[
    x \mapsto \phi(x) - \nabla \phi \cdot (x - \iota_Q^*(m)).
\]
Here, $\nabla \phi \cdot (x - \iota_Q^*(m)) = \sum_{j=1}^q \frac{\partial \phi}{\partial x_j} (x_j - \iota_Q^*(m)_j)$, where $x = (x_1, ..., x_q)$ are the coordinates on $\mathfrak{t}_{Q}^*$ induced by an arbitrary basis $\{\xi_{1},\dots,\xi_{q}\}$ of $\mathfrak{t}_{Q}$ and $\iota_Q^*(m) = (\iota_Q^*(m)_1, ..., \iota_Q^*(m)_q)$ in these coordinates. 
This real-valued function attains a unique maximum at $x=\iota_Q^*(m)$ and the critical point $x=\iota_Q^*(m)$ is non-degenerate.
\begin{lemma}\label{lem-img} Let $m \in P \cap \fot^{*}_{\ZZ}$. 
For all $t >0$, we have 
\begin{equation}\label{eq: action of weight-m section}
    e^{t\hat{\varphi}} \sigma^{m}= e^{t \mu_Q^* \left(\phi - \nabla \phi \cdot (x - \iota_Q^*(m)) \right)} \sigma^{m}.
\end{equation}
\end{lemma}
 
\begin{proof} 
Pick a basis $\{\xi_{1},\dots,\xi_{q}\}$ of $\mathfrak{t}_{Q}$ and write $\iota_Q^*(m) = (\iota_Q^*(m)_1, ..., \iota_Q^*(m)_q)$ in the induced linear coordinates $(x_1, ..., x_q)$ on $\mathfrak{t}_{Q}^*$. Since $\sigma^{m}$ is of weight $m$ with respect to the $T^{n}$-equivariant structure on $L$ determined by the moment map $\mu$, for each $j = 1, ..., q$, the infinitesimal action of $\xi_{j}$ on $\sigma^{m}$ is given by
\[
    -i\nabla_{\xi_{j}^{\#}} \sigma^{m} + (x_{j} \circ \mu_Q) \sigma^{m} = \iota_Q^*(m)_j \sigma^{m},
\]
where $\xi_{j}^{\#}$ is the fundamental vector field generated by $\xi_{j}$. Equivalently, we have
\[
    -i\nabla_{\xi_{j}^{\#}} \sigma^{m} = -\mu_Q^*(x_{j} - \iota_Q^*(m)_j) \sigma^{m}.
\]
Since the Hamiltonian vector field associated with the function $\varphi = \phi \circ \mu_Q$ can be expressed as $X_{\varphi}
=
\sum_{j=1}^{q} \mu_Q^* \left( \frac{\partial\phi}{\partial x_{j}} \right) \xi_{j}^{\#}$, 
we obtain
\begin{align*}
    \hat{\varphi} (\sigma^{m})
    = & \sum_{j=1}^{q} \mu_Q^* \left( \frac{\partial\phi}{\partial x_{j}} \right) \cdot (-i \nabla_{\xi_{j}^{\#}} \sigma^{m}) + \varphi \sigma^{m}\\
    = & \mu_Q^* \left(\phi - \nabla \phi \cdot (x - \iota_Q^*(m)) \right) \sigma^{m}.
\end{align*}
The function $\mu_Q^*\left( \phi - \nabla \phi \cdot (x - \iota_Q^*(m)) \right)$
is $T^{n}$-invariant, and is hence annihilated by $X_{\varphi}$. It implies that multiplication by this function commutes with the operator $\hat{\varphi}$. Therefore, by induction, for any non-negative integer $k$, we have:
$$\hat{\varphi}^{k}(\sigma^{m})
=
\Bigl( \mu_Q^*\left( \phi - \nabla \phi \cdot (x - \iota_Q^*(m)) \right) \Bigr)^{k}\sigma^{m}.$$
Consequently, we obtain the desired formula \eqref{eq: action of weight-m section}.
\end{proof}

By Theorem \ref{imaginary-flow},
 $\{\delta_{K,t}^{m}:=e^{t\hat{\varphi}}\delta_{K}^{m}\}_{m\in P\cap \mathfrak{t}_{\ZZ}^{*}}$ form a basis of $\shH_{K,t}$,
where $\delta_{K,t}^{m}$ is defined by
\begin{equation}\label{dis-basis}
\delta_{K,t}^{m}(\tau)
:=
\int_{\check{M}^{m}} \bigl\langle \sigma_{t}^{m},\tau\bigr\rangle\,
\vol_{K}^{m}
\end{equation}
for every test section $\tau\in \Gamma_c(M,L^{-1})$, where $\sigma_{t}^{m}$
is given by \eqref{eqn:sigma_t^m} and $\vol_K^{m}$ is the volume form given in Proposition \ref{decomp-vol}. Notice that $\mathring{M}^{m}$ is an open dense subset of $\check{M}^{m}$, hence
\begin{equation}\label{dis-basis2}
\delta_{K,t}^{m}(\tau)
= \int_{\mathring{M}^{m}} \bigl\langle \sigma_{t}^{m},\tau\bigr\rangle\,
\vol_{K}^{m}.
\end{equation}
Thus, both $\shH_{\RR}$ and $\shH_{K,t}$ admit canonical bases indexed by the
integral points of the moment polytope, namely
$\{\delta^{m}\}_{m\in P\cap \mathfrak{t}_{\ZZ}^{*}}$ and
$\{\delta_{K,t}^{m}:=e^{t\hat{\varphi}}\delta_{K}^{m}\}_{m\in P\cap \mathfrak{t}_{\ZZ}^{*}}$, respectively. In particular,
$$
\dim(\shH_{K,t})=\dim(\shH_{\RR}) \quad \text{for all $t\ge 0$}.
$$
We are now in a position to show that ``$\lim_{t\rightarrow \infty}\shH_{K,t}=\shH_{\RR}$'' in the
following sense:
$$
\Span_{\CC}\Bigl\{\lim_{t\to\infty} C_t^m\,\delta_{K,t}^m\Bigr\}_{m\in P\cap \mathfrak{t}_{\ZZ}^{*}}
\;=\;
\Span_{\CC}\{\delta^m\}_{m\in P\cap \mathfrak{t}_{\ZZ}^{*}},
$$
where for each $m\in P\cap \mathfrak{t}_{\ZZ}^{*}$ and $t>0$, the normalization constant $C_t^m$ is
\begin{equation}\label{eq-famconst}
C_t^m
=
\Bigl\| e^{-t \mu_Q^*\left(\phi - \nabla \phi \cdot (x-i_Q^*(m)) \right)} \Bigr\|_{L^1}^{-1}.
\end{equation}
Here, $\| \, \cdot \, \|_{L^1}$ denotes the $L^1$-norm on functions on $\check{M}^{m}$ with respect to the volume form $\vol_{K}^{m}$.

\begin{theorem}\label{thm-hlim}
Under assumption $(*)$, for any $m \in P \cap \fot^{*}_{\ZZ}$, the family of
distributional sections $\{C_{t}^{m}\delta_{K,t}^{m}\}_{t>0}\subset
\Gamma_{c}(M,L^{-1})'$ converges weakly to $C^{m}\delta^{m}$, that is: for
every test section $\tau\in \Gamma_{c}(M,L^{-1})$, one has
\begin{equation}
\lim_{t\to\infty} C_{t}^{m}\delta_{K,t}^{m}(\tau)= C^{m} \delta^{m}(\tau),
\end{equation}
where $C_{t}^{m}$ is as in \eqref{eq-famconst} and $C^{m} := |\sigma^{m}|(m)$.
\end{theorem}

\begin{proof}
Recall that $\{\sigma^{m}\}_{m\in P\cap \mathfrak{t}_{\ZZ}^{*}}$ denotes a basis of the quantum Hilbert space $\shH_{J}$ associated with $\shP_{J}$ such each each $\sigma^{m}$ is of weight $m$. Also recall from \cite[Theorem~3.7]{LW3} that the family $\{\sigma_{t}^{m}:=e^{t\hat{\varphi}}\sigma^{m}\}_{m\in P\cap \mathfrak{t}_{\ZZ}^{*}}$
is a basis of the quantum Hilbert space $\shH_{t}$ associated with
$\shP_{t}$. Now, by \eqref{dis-basis2}, we obtain
\begin{align*}
\left( \lim_{t \rightarrow \infty}{C_{t}^{m}}\delta_{K,t}^{m} \right)(\tau)
=\lim_{t \rightarrow \infty}{C_{t}^{m}} \int_{\mathring{M}^{m}}
\langle \sigma_{t}^{m}, \tau \rangle  \vol_{K}^{m}.
\end{align*}
Using the principal bundle structure of $\mu \vert_{\mathring{M}^{m}}: \mathring{M}^{m} \to \mathring{P}^{m}$ (see Lemma \ref{prin-bdl}) and the decomposition of the volume form $\vol_{K}^{m}$ in Proposition \ref{decomp-vol}, we may express the integral as an iterated integral over the base $P^{m}$ and the torus fiber:
\begin{align*}
    \int_{\mathring{M}^{m}} \langle \sigma_{t}^{m}, \tau \rangle  \vol_{K}^{m} = & \int_{\mathring{M}^{m}} e^{t \mu_{Q}^*\left( \phi - \nabla \phi \cdot (x - \iota_Q^*(m)\right)} \langle \sigma^{m}, \tau \rangle \cdot \vol_{\widetilde{T}} \wedge (\mu \vert_{\mathring{M}^{m}})^*(\vol_{P^{m}})\\
    = & \int_{\mathring{P}^{m}} e^{t \left( \phi - \nabla \phi \cdot (x - \iota_Q^*(m)\right) \circ \iota_{Q}^*} \left( \int_\mu \langle \sigma^{m}, \tau \rangle \right) \vol_{P^{m}}\\
    = & \int_{P^{m}} e^{t \left( \phi - \nabla \phi \cdot (x - \iota_Q^*(m)\right) \circ \iota_{Q}^*} \left( \int_\mu \langle \sigma^{m}, \tau \rangle \right) \vol_{P^{m}}.
\end{align*}
Here, $\int_\mu \langle \sigma^{m}, \tau \rangle$ is defined in \eqref{eqn-fib-int}, and is continuous on $P^{m}$. Similarly,
\begin{align*}
(C_{t}^{m})^{-1} = & \int_{\mathring{M}^{m}} e^{t \mu_Q^*\left( \phi - \nabla \phi \cdot (x - \iota_Q^*(m)\right)} \vol_{K}^m\\
= & \int_{P^{m}} e^{t \left( \phi - \nabla \phi \cdot (x - \iota_Q^*(m)\right) \circ \iota_Q^*} \vol_{P^{m}} =: \Bigl\| e^{t \left( \phi - \nabla \phi \cdot (x - \iota_Q^*(m)\right) \circ \iota_Q^*} \Bigr\|_{1}.
\end{align*}

By transversality, the restriction $\iota_Q^* |_{P^{m}}: P^{m} \to \fot_{Q}^*$ is injective, so the restriction of
\[
    \left( \phi - \nabla \phi \cdot (x - \iota_Q^*(m)\right) \circ \iota_Q^*
\]
to a neighbourhood of polytope $P^{m}$ in its affine hull has a unique maximum at $m \in P^{m}$ and $m$ is a non-degenerate critical point. Applying Lemma \ref{conv-refine} and \eqref{rescale} then yields
\begin{align*}
    \left(\lim_{t \rightarrow \infty}{C_{t}^{m}}\delta_{K,t}^{m} \right)(\tau) = & 
    \lim_{t \rightarrow \infty} \int_{P^{m}} \frac{e^{t \left( \phi - \nabla \phi \cdot (x - \iota_Q^*(m)\right) \circ \iota_{Q}^*}}{\Bigl\| e^{t \left( \phi - \nabla \phi \cdot (x - \iota_Q^*(m)\right) \circ \iota_{Q}^*} \Bigr\|_{1}} \left( \int_\mu \langle \sigma^{m}, \tau \rangle \right) \vol_{P^{m}}\\
    = & \left( \int_\mu \langle \sigma^{m}, \tau \rangle \right) (m)\\
     = & | \sigma^m |(m) \cdot \delta^m(\tau),
\end{align*}
which establishes weak convergence.
\end{proof}

\phantomsection

\appendix

\section{Lemmas for the proof of Theorem \ref{thm-hlim}}
\label{appendix}

Let $m \in P \cap \fot^{*}_{\ZZ}$. The following lemma holds.

\begin{lemma}\label{prin-bdl}
    There exists a unique subtorus $H < T^{n}$ with the following properties:
    \begin{itemize}
        \item $\mathring{M}^{m}$ is precisely the set of points in $M^{m}$ whose $T^{n}$-stabilizer equals $H$;
        \item $H$ is contained in the $T^{n}$-stabilizer of every point of $M^{m}$.
    \end{itemize}
    Consequently, $\mu \vert_{\mathring{M}^{m}}: \mathring{M}^{m} \to \mathring{P}^{m}$ is a principal $T^{n} / H$-bundle.
\end{lemma}
\begin{proof}
    Let $F$ be the intersection of all faces of $P$ containing $P^{m}$; then $F$ is itself a face of $P$. Denote the relative interior of $F$ by $\mathring{F}$.
    By Lemma \ref{rel-int} (to be proved below), the relative interior $\mathring{P}^{m}$ of $P^{m}$ satisfies
    $$\mathring{P}^{m} = P^{m} \cap \mathring{F}.$$
    By the orbit–face correspondence, points of $M$ lying over $\mathring{F}$ have minimal $T^{n}$-stabilizer among all points lying over $F$. It follows that $\mathring{M}^m = \mu^{-1}(\mathring{P}^{m})$ consists exactly of those points in $M^{m} = \mu^{-1}(P^{m})$ whose $T^{n}$-stabilizer is minimal. In particular, this stabilizer is constant on $\mathring{M}^{m}$; denote it by $H$. By minimality, $H$ is contained in the stabilizer of every point of $M^{m}$, and uniqueness is immediate.

    Since the $T^{n}$-action on $\mathring{M}^{m}$ has constant stabilizer $H$, the induced action of $T^{n}/H$ is free. Therefore, $\mu \vert_{\mathring{M}^{m}}$ is a principal $T^{n} / H$-bundle.
\end{proof}

\begin{lemma}\label{rel-int}
    Let $P^{m}, \mathring{P}^{m}$ and $\mathring{F}$ be as above. Then $\mathring{P}^{m} = P^{m} \cap \mathring{F}$.
\end{lemma}
\begin{proof}
    First, we claim that
    $$P^{m} = \operatorname{Aff}(P^{m}) \cap F \subset \operatorname{Aff}(F),$$
    where $\operatorname{Aff}(P^{m})$ and $\operatorname{Aff}(F)$ denote the affine hulls of $P^{m}$ and $F$, respectively. The reason is as follows. Clearly, $P^{m} \subset \operatorname{Aff}(P^{m}) \cap F$. On the other hand, $P^{m}$ is contained in the affine subspace $\{ x \in \mathfrak{t}^*: \iota_K^*(x) = \iota_K^*(m) \}$, hence $\operatorname{Aff}(P^{m}) \subset \{ x \in \mathfrak{t}^*: \iota_K^*(x) = \iota_K^*(m) \}$. Then
    \[
        \operatorname{Aff}(P^{m}) \cap F \subset \{ x \in \mathfrak{t}^*: \iota_K^*(x) = \iota_K^*(m) \} \cap P = P^{m},
    \]
    therefore our claim holds.
    
    Now, we show that $P^{m} \cap \mathring{F} \subset \mathring{P}^{m}$. If $x \in P^{m} \cap \mathring{F}$, then there exists an open set $B \subset \operatorname{Aff}(F)$ with $x \in B \subset F$. It follows that $\operatorname{Aff}(P^{m}) \cap B$ is open in $\operatorname{Aff}(P^{m})$ and contained in $\operatorname{Aff}(P^{m}) \cap F = P^{m}$, hence $x \in \mathring{P}^{m}$.

    Finally, we prove $\mathring{P}^{m} \subset P^{m} \cap \mathring{F}$ by contradiction. If $x \in \mathring{P}^{m}$ and $x \not\in P^{m} \cap \mathring{F}$, then $x$ lies on a facet $F'$ of $F$. By the minimality of $F$, $P^m \not\subset F'$. Pick $x' \in P^m$ such that $x' \not\in F'$, and choose an affine function $f$ on $\operatorname{Aff}(F)$ and $c \in \mathbb{R}$ such that $f \leq c$ on $F$ and $F' = f^{-1}(\{c\}) \cap F$. Since $x' \in F \backslash F'$, $f(x') < c$. Now, define a vector $\epsilon = x' - x$. By the condition $x \in \mathring{P}^m$, we can choose a sufficiently small constant $\lambda > 0$ such that $x - \lambda \epsilon \in P^m$. However, since $f$ is an affine function, we must have $f(x - \lambda \epsilon) > c$, implying that $x - \lambda \epsilon \not\in F$. It contradicts the condition that $P^m \subset F$, concluding the proof.
\end{proof}

\begin{lemma}\label{conv-refine}
    Let $\triangle \subset \mathbb{R}^{\ell}$ be an $\ell$-dimensional polytope and $\psi$ be a smooth function on a neighbourhood of $\triangle$ with a unique minimum at $x_0 \in \triangle$ such that the critical point $x_0$ of $\psi$ is non-degenerate. For $t > 0$, define
    \begin{equation}
        \lVert e^{-t\psi} \rVert_{1} := \int_{\triangle} e^{-t\psi} dx,
    \end{equation}
    where $dx$ is the Euclidean volume form on $\mathbb{R}^{\ell}$. Then for any continuous function $f$ on $\triangle$,
    \begin{equation}
        \lim_{t \rightarrow \infty} \frac{1}{\lVert e^{-t\psi} \rVert_{1}} \int_{\triangle} f e^{-t\psi} dx = f(x_0).
    \end{equation}
\end{lemma}
\begin{proof}
    The proof is standard and follows the same argument as that of \cite[Lemma 3.7]{BFMN}. We include the details for completeness.

    Since $\psi$ is smooth and has a minimum at $x_0$, we know that $\nabla \psi (x_0) = 0$ and the Hessian $H(x_0) := D^2 \psi(x_0)$ is positive semidefinite. Using a second-order Taylor expansion around $x_0$:
    $$ \psi(x) = \psi(x_0) + \tfrac{1}{2} (x - x_0)^T H(x_0)(x - x_0) + o \lVert x - x_0 \rVert^2.$$
    Since $x_0$ is a non-degenerate critical point of $\psi$, the Hassian $H(x_0)$ is actually positive definite. So we have:
    $$\exists c_{min}, c_{max} > 0 \text{ such that } c_{min} \lVert x - x_0 \rVert^2 \leq (x - x_0)^T H(x_0) (x - x_0) \leq c_{max} \lVert x - x_0 \rVert^2.$$
    By the meaning of the $o (\lVert x - x_0 \Vert^2)$-term, there exist constants $c, \alpha > 0$ such that, for $x$ sufficiently close to $x_0$,
    \begin{equation}
        \label{inequ}
        \psi(x_0) + c \lVert x - x_0 \rVert^2 \leq \psi(x) \leq \psi(x_0) + \alpha \lVert x - x_0 \rVert^2.
    \end{equation}

    We show that the functions
    $$ \zeta_s := \frac{e^{-s\psi}}{\lVert e^{-s\psi} \rVert_{1}}.$$
    form a Dirac sequence. (We actually show convergence as measures on $\triangle$.) It is clear from the definition that $\zeta_s > 0$ and $\lVert \zeta_s \rVert _{1} = 1$, so it remains to show that the norms concentrate around the minimum, that is, given any $\varepsilon, \varepsilon' > 0$ we have to find a number $s_0$ such that
    $$\forall s \geq s_0: \int_{\Delta \cap B_\varepsilon(x_0)} \zeta_s(x) dx \geq 1 - \varepsilon'.$$
    Choose $r > 0$ such that inequality \eqref{inequ} holds for all $x \in B_r(x_0)$ and $- c \varepsilon^2 + \alpha r^2 < 0$. Observe that
    $$ \lVert e^{-s\psi} \rVert_{1} = \int_{\triangle} e^{-s\psi(x)} dx \geq \int_{\Delta \cap B_r(x_0)} e^{-s\psi(x)} dx \geq \operatorname{Vol}(B_r(x_0)) \cdot e^{-s\psi(x_0) - s\alpha r^2}.$$
    On the other hand,
    $$\int_{\triangle \backslash B_\varepsilon(x_0)} e^{-s\psi(x)} d x \leq \int_{\triangle \backslash B_\varepsilon(x_0)} e^{-s\psi(x_0) - sc \lVert x - x_0 \rVert^2} d x \leq \operatorname{Vol}(\triangle) \cdot e^{-s\psi(x_0) - c \varepsilon^2}.$$
    Therefore,
    $$ \int_{\triangle \backslash B_\varepsilon(x_0)} \zeta_s(x) d x \leq \frac{\operatorname{Vol}(\triangle) \cdot e^{- sc \varepsilon^2 + s\alpha r^2 }}{\operatorname{Vol}(B_r(x_0))}.$$
    The right hand side goes to zero as $s \to \infty$ and the result follows.
\end{proof}

\providecommand{\bysame}{\leavevmode\hbox to3em{\hrulefill}\thinspace}
\providecommand{\MR}{\relax\ifhmode\unskip\space\fi MR }
\providecommand{\MRhref}[2]{%
  \href{http://www.ams.org/mathscinet-getitem?mr=#1}{#2}
}
\providecommand{\href}[2]{#2}

\end{document}